\documentclass[12pt, article]{amsart}

\usepackage{amsmath,amssymb,amsthm}
\usepackage{mathrsfs}
\usepackage{geometry}
\usepackage{graphicx}
\usepackage{hyperref}
\usepackage{enumitem}

\newtheorem{theorem}{Theorem}[section]
\newtheorem{proposition}[theorem]{Proposition}
\newtheorem{corollary}[theorem]{Corollary}

\theoremstyle{definition}

\newtheorem{remark}[theorem]{Remark}
\begin{document}
\title[ ]{A Comparison of some Weighted Numerical Radii of Hilbert Space Operators}

 \author[M. Khosravi]{Maryam Khosravi}
\address{Department of Pure Mathematics, Faculty of Mathematics and Computer, Shahid Bahonar University of Kerman, Kerman, Iran} \email{khosravi$_-$m@uk.ac.ir; khosravi$_-$m2000@yahoo.com}

\author[A.
Sheikhhosseini ]{Alemeh
Sheikhhosseini}
\address{Department of Pure Mathematics, Faculty of Mathematics and Computer, Shahid Bahonar University of Kerman, Kerman, Iran} \email{sheikhhosseini@uk.ac.ir; hosseini8560@gmail.com}

\date{}



\maketitle

\begin{abstract}
In recent years, several generalizations of the numerical radius for bounded linear operators on  Hilbert spaces have been introduced and extensively studied. These generalizations provide refined tools for investigating operator inequalities and spectral properties. In this paper, we investigate several generalized numerical radii and establish relationships among them.

We also establish several relationships among these generalized numerical radii and show that the $(s,t)$-weighted numerical radius can be represented as a rescaled form of the $t$-weighted numerical radius.
\end{abstract}

\section{Introduction and preliminaries}
Let $ \mathcal B(H) $ denote the C$^{*}$-algebra of all bounded linear operators on a Hilbert space $ \mathcal{H}. $ For $A\in\mathcal B(H)$, the  numerical range is  defined as follows.
$$ W(A)=\lbrace \langle Ax, x \rangle: x \in \mathcal{H},   \Vert x \Vert=1   \rbrace.  $$
The following standard properties of the numerical range are well known (see, for example, \cite{gust}):\\

(i) $ W(\alpha I + \beta A)= \alpha +\beta W(A)  $ for  $ \alpha, \beta \in \mathbb{C}; $\\

(ii)  $W(A^{*})= \lbrace  \overline{\lambda}: \lambda \in W(A) \rbrace;$\\

(iii) $ W(U^{*}AU)=W(A) $ for any unitary operator $ U. $\\

The classical numerical radius  of an operator  $ A $, denoted by $ \omega(A) $, is defined as
\[
\omega(A)=\sup\{ |\langle Ax,x\rangle| : x\in H,\ \|x\|=1 \}.
\]
It is well known that $\omega(\cdot)$ defines a norm on $\mathcal B(H)$ and satisfies the inequality
\begin{equation}\label{eq1}
\tfrac12\|A\|\le \omega(A)\le \|A\|.
\end{equation}
Moreover, if $A^{2}=0, $  then $ \omega (A)=\frac{1}{2} \Vert A \Vert$ and if $ A $ is normal, then
$ \omega(A)=\Vert A \Vert. $\\

The numerical radius  plays  a fundamental role in operator theory and is closely related  to the operator norm and spectral properties.
 Numerous refinements and generalizations of the numerical radius inequalities have been studied in the literature (see, for example, \cite{q-num,Sattari, sheyb,zam_2021}).

\section{The $ t$-weighted numerical radius }

In 2007, Yamazaki \cite{yam} established the following identity for the numerical radius in terms of the operator norm.
\begin{align}\label{yam}\omega(A)=\sup_{\theta}\|\Re(e^{i\theta}A)\|.\end{align}
Later, the following equivalent presentation was derived in \cite{kit}.
\begin{theorem} For $A\in\mathbb{B}(\mathcal H)$,
\begin{align}\label{kit}\omega(A)=\sup_{\alpha^2+\beta^{2}=1}\|\alpha \Re A+\beta\Im A\|, \end{align}
where the supremum is taken over all such real numbers.
\end{theorem} 
In \cite{sheikh},  Sheikhhosseini, Khosravi and Sababheh  
  defined  the $ t- $weighted real  and  imaginary parts  of  $A \in\mathbb{B}(\mathcal H)$ as:
 \begin{equation}\label{def1}
 \Re_t (A) \,  :=tA+ (1-t)A^*\qquad \text{and}\qquad \Im_t(A) \, :=\frac{tA-(1-t)A^*}{i}.
 \end{equation}
The following identities follow directly from the definitions of $\Re_{t}$ and $\Im_{t}$.
\begin{proposition}
Let $A\in\mathcal{B}(\mathcal{H})$ and let $0\leq t \leq 1.$ Then
$$\Im_{t}(iA)  = \Re_{t} (A), ~ \Re_{t} (iA) = -\Im_{t}(A),$$
\begin{equation*}
\Re_{t} (A) =\Re (A) +(2t-1)i \Im(A),
\end{equation*}
\begin{equation*}
\Im_{t}(A) =\Im(A)- (2t-1)i \Re (A),
\end{equation*}
and
\begin{equation}\label{eq000}
\Re_{t}(A)+i\Im_{t}(A)=2t A.
\end{equation}
\end{proposition}
 Using to the definition of $ \Re_{t}(.) $ in   (\ref{def1}) and  relation (\ref{yam}), they introduced  the following definition for  the weighted numerical radius by 
 $$\omega^{\Re}_t(A) \, :=\sup_{\theta}\|\Re_t(e^{i\theta}A)\|.$$
 In particular, $$\omega^{\Re}_0(A)=\omega^{\Re}_1(A)=\|A\|,\qquad\omega^{\Re}_{1/2}(A)=\omega(A).$$
 
 The following theorem summarizes several fundamental properties of $ \omega^{\Re}_t (.). $
 \begin{theorem}\cite{sheikh}
 Let $A\in\mathbb{B}(\mathcal H)$ and $0\leq t\leq1$. Then
 \begin{enumerate}
 \item $\omega^{\Re}_t(.)$ defines a norm on $\mathbb{B}(\mathcal H)$ which is equivalent to the numerical range and operator norm;
 \item $\omega^{\Re}_t(A)=\omega^{\Re}_t(A^*)=\omega^{\Re}_{1-t}(A)$;
 \item $\max\{t,1-t\}\|A\|\leq\omega^{\Re}_t(A)\leq\|A\|$;
 \item If $A$ is a normal operator, then $\omega^{\Re}_t(A)=\|A\|$;
 \item $\omega^{\Re}_t(A)\leq \omega^{\Re}_s(A)$ if and only if $|t-\frac{1}{2}|\leq|s-\frac{1}{2}|$.
   \end{enumerate}
   \end{theorem}
 The following properties follow directly from the definition of $ \omega^{\Re}_t(.) $:
  \begin{proposition}\cite{sheikh}
  Let $A\in\mathbb{B}(\mathcal H)$ and $0\leq t\leq1$. Then
   \begin{enumerate}
  \item $\omega^{\Re}_t(U^{*}AU)=\omega^{\Re}_t(A) $ for every unitary  $ U \in\mathbb{B}(\mathcal H); $
  \item $ \omega^{\Re}_t \left( \begin{bmatrix}
  A & 0\\
  0& B
\end{bmatrix}\right) =\max \lbrace \omega^{\Re}_t( A), \omega^{\Re}_t( B)\rbrace.  $
\end{enumerate}
\end{proposition}   
The following representation extends identity \eqref{kit} to the weighted setting.
\begin{align*}\omega^\Re_t(A)&=\sup_{\alpha^2+\beta_2=1}\|\alpha \Re_t (A)+\beta\Im_t (A)\|\\
&=\sup_{\alpha^2+\beta_2=1}\|\alpha \Re (A)+\beta\Im (A)+i(2t-1)(\alpha \Im (A)-\beta \Re(A))\|.
\end{align*}
The following theorem was proved in \cite{sheikh}; for completeness,  we provide a more simple proof for it. 
\begin{theorem}\label{2.5}
Let $A\in\mathbb{B}(\mathcal H)$ and $0\leq t\leq1$. Then
$$\omega(A)\leq\omega^\Re_t(A)\leq2R\ \omega(A),$$
where $R=\max\{t,1-t\}$.
\end{theorem}
\begin{proof} By the triangle inequality, we obtain
\begin{align*}
\|\alpha \Re_t (A)+\beta\Im_t (A)\|&=\|\alpha \Re (A)+\beta\Im (A)+i(2t-1)(\alpha \Im (A)-\beta \Re(A))\|\\
&\leq \|\alpha \Re (A)+\beta\Im (A)\|+|2t-1|\|\alpha \Im (A)-\beta \Re(A)\|.
\end{align*}
Taking the supremum over all real numbers $\alpha,\beta$ with $\alpha^2+\beta^2=1$, completes the proof.
\end{proof}
In the following theorem,  in view of  a  refinement and reverse of the classical Jensen inequality \cite{jen},   we present upper and lower bounds for the difference between the two sides of this inequality in terms of $ \omega^{\Re}_t(.). $

\begin{theorem}\label{t1}
Let $A\in\mathcal B(H).$ Then for every $ 0< t < 1, $

$$ \dfrac{\| A \| - \omega^{\Re}_t(A)}{2R}  \leq \| A \| - \omega (A) \leq \dfrac{\| A \| - \omega^{\Re}_t(A)}{2r},  $$
where  $r=\min\{t,1-t\}$ and $R=\max\{t,1-t\}$.
\end{theorem}

\begin{proof}
In \cite{jen}, the authers   showed that if  $ f:[0, 1] \rightarrow \mathbb{R} $ is a convex function, then for every $ 0 < t < 1, $
$$ \dfrac{(1-t) f(0)+t f(1)-f(t)}{2R}\leq \dfrac{f(0)+f(1)}{2}-f(\dfrac{1}{2}) \leq  \dfrac{(1-t) f(0)+t f(1)-f(t)}{2r}, $$
where $r=\min\{t,1-t\}$ and $R=\max\{t,1-t\}$. Employing this for $ f(t)=\omega^{\Re}_t(A) $ and the fact that $ \omega^{\Re}_{0}(A)= \omega^{\Re}_{1}(A) = \| A \|, \omega^{\Re}_{1/2}(A)=\omega(A), $
we reach the desired result.
\end{proof}
\begin{corollary}
Let $A\in\mathcal B(H).$ Then for every $ 0< t < 1, $

$$ (1-2R) \| A \| +2R \omega(A) \leq \omega^{\Re}_t(A) \leq (1-2r) \| A \| +2r \omega(A) ,  $$
where  $r=\min\{t,1-t\}$ and $R=\max\{t,1-t\}$.
\end{corollary}
\begin{corollary}
Let $A\in\mathcal B(H).$ Then 

$$ \left(3 \omega(A)- \| A \| \right) \leq 4\int_{0}^{1/2} \omega^{\Re}_t(A) dt  \leq \left( \| A \| + \omega(A) \right)  $$
and
$$ \left(3 \omega(A)- \| A \| \right) \leq 2\int_{0}^{1} \omega^{\Re}_t(A) dt  \leq \left( \| A \| + \omega(A) \right),  $$

\end{corollary}
\begin{proposition}
Let $A\in\mathcal B(H).$ Then for every $ 0< t < 1, $ we have
$$ \omega (A) \leq  t \omega^{\Re}_{\frac{t}{2}}(A) +(1-t) \omega^{\Re}_{\frac{1+t}{2}} (A) \leq 
\int_{0}^{1} \omega^{\Re}_{t}(A) dt \leq \frac{1}{2}( \omega^{\Re}_t(A)+ \| A \| ) \leq \| A \|.$$
\end{proposition}

\begin{proof}
Using the following refinement of the Hermite–Hadamard inequality \cite{he-ha},
$$ f(\dfrac{1}{2}) \leq l(t) \leq \int_{0}^{1}  f(t) dt \leq L(t) \leq \dfrac{f(0)+f(1)}{2}$$
for a convex function $ f:[0, 1] \rightarrow \mathbb{R} $ and $ 0 < t < 1,  $ where
$$  l(t)= t f(\frac{t}{2})+(1-t) f(\frac{1+t}{2}) $$
and 
$$  L(t)=\dfrac{1}{2} \left( f(t)+t f(0)+(1-t)f(1)     \right).$$
 Employing this for $ f(t)=\omega^{\Re}_t(A) $ and the fact that $ \omega^{\Re}_{0}(A)= \omega^{\Re}_{1}(A) = \| A \|$ and $ \omega^{\Re}_{1/2}(A)=\omega(A), $
we reach the desired result.
\end{proof}

\section{A second weighted numerical radius}
In \cite{sab}, the authors, in view  of the definitions of $ \Re_t(A),  \Im_t(A)$ in relation (\ref{def1}) and equation (\ref{eq000}),    introduced a different generalization  of  the numerical radius in the following way.

For $A\in\mathbb{B}(H)$, define  $A_t$ by $A_t=\Re_t(A)+i\Im_{1-t}(A)=(1-2t)A^*+A.$ An  associated weighted numerical radius $ \omega_t(.) $ and  weighted operator norm $  \Vert.\Vert_{t}  $ are defined by
 $$\omega_t(A) \, :=\omega(A_t),   \qquad  \|A\|_t \, :=\|A_t\|.$$ 
 It follows immediately that
$$A_0=2\Re A,  \qquad  A_{1/2}=A,\qquad A_1=2i \Im A.$$
This weighted numerical radius differs significantly from $\omega^\Re_t(.)$ defined in the previous section.
In addition, $(A_t)^*=A^*_t$ and $(\lambda A)_t=\lambda A_t$ for every  $\lambda \in \mathbb{R}$  but does not hold in general for $ \lambda \in \mathbb{C}. $  Regarding the properties of $\omega$, we can state the following proposition.
\begin{proposition} Let $A,B \in\mathbb{ B}(\mathcal H)$ and let $0 \leq t \leq 1$. Then
\begin{enumerate}
\item $\omega_t(A)=\omega_t(A^*)$;
\item $\omega_t(\lambda A)=|\lambda|\omega_t(A)$ for every real number $\lambda$;
\item $\omega_t(A+B)\leq \omega_t(A)+\omega_t(B)$;
\item The function $f (t) = \omega_t (A)$ is convex on the interval $[0, 1]$.
\end{enumerate}
\end{proposition}
\begin{remark}
Since $\omega_t(.)$ and $\|.\|_t$ are not homogeneous over $\mathbb{C}$, they  do not define  norms on $ \mathbb{B}(\mathcal H). $ 
\end{remark}

The following theorem is partly stated in \cite{sab}.
\begin{theorem}\label{3.3}
Let $A\in\mathbb{B}(\mathcal H)$ and $0\leq t\leq1$. Then
$$2r\ \omega(A)\leq\omega_t(A)\leq2R\ \omega(A),$$
where $r=\min\{t,1-t\}$ and $R=\max\{t,1-t\}$.
\end{theorem}
\begin{proof}
By triangle inequality and properties of $\omega$, we have
$$\omega_t(A)=\omega(A+(1-2t)A^*)\leq\omega(A)+|1-2t|\omega(A^*)=(1+|1-2t|)\omega(A),$$
and
$$\omega_t(A)=\omega(A+(1-2t)A^*)\geq\big|\omega(A)-|1-2t|\omega(A^*)\big|=\big|1-|1-2t|\big|\omega(A).$$
\end{proof}
We conclude this section with the following refinement involving the weighted numerical radius.

Using weighted numerical radius, the following refinement of the first inequality in (\ref{eq1}) is obtained:
\begin{proposition}\cite{sab}
Let $A\in\mathbb{B}(\mathcal H).$  Then
$$  \frac{1}{2}\Vert  A \Vert \leq  \frac{1}{2} \left(   \Vert \Re(A)  \Vert +  \Vert \Im (A) \Vert  \right) \leq
\frac{2}{3} \left(   \int_{0}^{\frac{1}{2}} [ \omega_{t}(A)+ \omega((1-2t)A-A^{*})] dt  \right)  \leq \omega(A).$$
\end{proposition}

\section{Extensions of the numerical radius via arbitrary norms}
 Another generalization \eqref{yam} has been presented in \cite {O-F}, where the authors noticed the use of the usual operator norm in the identity $\omega(A)=\sup\limits_{\theta}\|\Re(e^{i\theta}A)\|.$ 
 They generalized the numerical radius by replacing the operator norm  with an arbitrary norm $N(.).$

\begin{align}\label{eq_abo_amer}
\omega_N(A)=\sup_{\theta}N\left(\Re(e^{i\theta}A)\right),
\end{align}
where $N(\cdot)$ is a given norm on $\mathcal{B}(\mathcal{H}).$ 
This definition has the following properties:
\begin{theorem}
Let  $N(.)$ be a norm on $\mathbb{B}(\mathcal H)$. Then,   for every $A\in \mathbb{B}(\mathcal H)$,
\begin{enumerate}
\item $\omega_N(.)$ is a norm on $\mathbb{B}(\mathcal H)$;
\item $\omega_N(A^*)=\omega_N(A)$.
\end{enumerate}
\end{theorem}
In general, $\omega_N$ is not unitarily invariant. In fact, if $N(.)$ is weakly unitarily invariant, then $\omega_N(U^*AU)=\omega_N(A)$ for each unitary operator $U$.

Also, Bottazzi and Conde  \cite{bot} derived the following theorem.

\begin{theorem}
For every $A\in \mathbb{B}(\mathcal H)$ and every norm $N(.)$,
\begin{enumerate}
\item $\omega_N(A)=\sup_{\alpha^2+\beta^2=1}N(\alpha\Re(A)+\beta\Im(A)$;
\item $\max\{\frac{1}{2}N(A),\frac{1}{2}N(A^*)\}\leq\omega_N(A)\leq\frac{1}{2}(N(A)+N(A^*))$.
\end{enumerate}
\end{theorem}
Combining the idea of \cite{O-F} and \cite{sheikh}, Zamani defined
$$\omega_{(N,t)}(A)=\sup_{\theta} N(\Re_t(e^{i\theta}A))\qquad (A\in\mathbb{B}(\mathcal H)),$$
for each $0\leq t\leq 1$. 
Similar to \eqref{kit}, we have
$$\omega_{(N,t)}(A)=\frac{1}{2}\sup_{\theta,\varphi\in\mathbb{R}}N(\Re_t((e^{i\theta}-e^{i\varphi})A)).$$

This notation defines a norm on $\mathbb{B}(\mathcal H)$ and
$$\max \{tN(A),(1-t)N(A^*)\} \leq\omega_{(N,t)}(A) \leq\max \{N(A), N(A^*)
\} .$$
In the following theorem, we state some properties of $\omega_{(N,t)}$.
\begin{theorem}
Let $N(.)$ be a norm and $0\leq t\leq1$. For each $A\in\mathbb{B}(\mathcal H)$,
\begin{enumerate}
\item $\omega_{(N,t)}(A^*)=\omega_{(N,1-t)}(A)$;
\item If $N(A)=N(A^*)$, then $\omega_{(N,t)}(A^*)=\omega_{(N,t)}(A)$;
\item $\omega_{(N,t)}(A)\leq\omega_{(N,s)}(A)$ if and only if $|t-\frac{1}{2}|\leq|s-\frac{1}{2}|$.
\end{enumerate}
\end{theorem}
In particular, he investigated the case of the Hilbert-Schmidt norm and obtained some bounds for $\omega_{(2,t)}$ where this symbol used for the case that $N$ is considered as Hilbert-Schmidt norm.

A similar concept was introduced in the framework of C$^*$-algebras, see \cite{c*}, and has been generalized in  several works. In particular, Gao and Hou \cite{gao} investigated  a special extension of numerical radius, see also \cite{mab}. Here, we restate them in an operator setting.

Let $s,t$ be two nonnegative real  numbers with $s+t>0$ and $A \in\mathbb{B}({\mathcal H})$. Define
$$\Re_{(s,t)}(A)=sA+tA^*,\qquad \Im_{(s,t)}(A)=\Re_{(s,t)}(-iA).$$
The $(s,t)$-weighted numerical radius is defined as
$$\omega_{(s,t)}(A)=\sup_\theta\|\Re_{(s,t)}(e^{i\theta}A)\|.$$
In the following result, we present some properties of $\omega_{(s,t)}$.
\begin{theorem}\label{4.4}
 Let $A\in\mathbb{B}(\mathcal H).$ Then
\begin{enumerate}
\item $\omega_{(s,t)}(A)=\omega_{(t,s)}(A^*);$
\item If $A$ is self-adjoint, then $(s+t)\omega(A)=\omega_{(s,t)}(A)=(s+t)\|A\|;$
\item $\omega_{(s,t)}$ is a norm on $\mathbb{B}(\mathcal H);$
\item $\max\{s,t\}\|A\|\leq\omega_{(s,t)}(A)\leq (s+t)\|A\|.$
\end{enumerate}
\end{theorem}
As a special case, let $\lambda\in\mathbb{R}$ and $t=\lambda-s$.
\begin{theorem}\label{4.5}
Let $A\in\mathbb{B}(\mathcal H).$ Then
\begin{enumerate}
\item $\omega_{(\frac{\lambda}{2},\frac{\lambda}{2})}(T)\leq\omega_{(s,\lambda-s)}(T)\leq \omega_{(\lambda,0)}(T)$;
\item $\omega_{(s,\lambda-s)}(T)\leq \omega_{(t,\lambda-t)}(T)$ if and only if $|s-\frac{\lambda}{2}|\leq|t-\frac{\lambda}{2}|;$
\item $\lambda\omega(T)\leq \omega_{(s,\lambda-s)}(T)\leq\lambda\|T\|.$
\end{enumerate}
\end{theorem}
\begin{theorem}
Let $A\in\mathbb{B}(\mathcal H).$ Then
$$\omega_{(s,t)}(A)=\frac{1}{2}\sup_{\theta,\varphi}\|\Re_{(s,t)}((e^{i\theta}-ie^{i\varphi})A)\|$$
and
\begin{align*}\omega_{(s,t)}(A)&=\sup_{\alpha^2+\beta^2=1}\|\alpha\Re_{(s,t)}(A)+\beta\Im_{(s,t)}(A)\|\\
&=\sup_{\alpha^2+\beta^2=1}\|(s+t) (\alpha\Re(A)+\beta \Im(A))+i(s-t)(\alpha\Im(A)-\beta\Re(A))\|.\end{align*}
\end{theorem}

\section{Relationships between generalized numerical radii }
First of all, we summarize some of the main properties of these generalizations in the following table.\\
\begin{center}
\begin{tabular}{|l|c|c|c|c|}
\hline
radius     & norm & unitarily invariant & coincides with $\omega$ when & value for hermitian \\\hline
$\omega^{\Re}_t$        & yes  & yes                 & $t=\frac{1}{2}$           & $\|.\|$                       \\\hline
$\omega_t$       & no   & no       & $t=\frac{1}{2}$           & $2(1-t)\|.\|$                 \\\hline
$\omega_{(s,t)}$ & yes  & yes    & $s=t=\frac{1}{2}$         & $(s+t)\|.\|$            \\\hline    
\end{tabular}\end{center}\vspace{3mm}
Applying Theorem \ref{2.5} and Theorem \ref{3.3}, it  easily follows that 
$$ \omega_t(A)\leq2R \omega^{\Re}_t(A),$$
where $R=\max\{t,1-t\}$.

The following theorem presents another comparison between these generalized numerical radii.
 \begin{theorem}
Let $A\in \mathcal B(H)$.
\begin{itemize}
\item If $0\leq t\leq\frac{1}{2}$, then 
$$\omega_t(A)\leq 2(1-t)\omega^{\Re}_{\frac{1}{2(1-t)}}(A).$$
\item If $\frac{1}{2}\leq t\leq1$, then 
$$\omega_t(A)\leq 2t\omega^{\Re}_{\frac{1}{2t}}(A).$$
\end{itemize}
\end{theorem}
\begin{proof} 
If $0\leq t\leq\frac{1}{2}$, then $0\leq1-2t\leq1$. Thus
\begin{align*}
\omega_t(A)&=\omega((1-2t)A^*+A)\\
&\leq\|(1-2t)A^*+A\|\\
&=2(1-t)\|\frac{1-2t}{2-2t}A^*+\frac{1}{2-2t}A\|\\
&\leq2(1-t)\|\Re_{\frac{1}{2-2t}}(A)\|\\
&\leq 2(1-t)\omega^{\Re}_{\frac{1}{2-2t}}(A).
\end{align*}

If $\frac{1}{2}\leq t\leq1$, then $0\leq2t-1\leq1$. Thus
\begin{align*}
\omega_t(A)&=\omega((1-2t)A^*+A)\\
&=\omega((2t-1)e^{i\pi}A^*+A)\\
&\leq\|(2t-1)e^{i\pi}A^*+A\|\\
&=2t\|\frac{2t-1}{2t}e^{i\pi}A^*+\frac{1}{2t}A\|\\
&\leq2t\|\Re_{\frac{1}{2t}}(e^{-i\pi/2}A)\|\\
&\leq 2t\omega^{\Re}_{\frac{1}{2t}}(A).
\end{align*}
\end{proof}

\begin{theorem}
Let $A\in\mathcal B(H)$ and $ 0 \leq t \leq 1. $ Then 
$$\sqrt{2}  \omega^{\Re}_{t}  (A) \leq \left( \Vert A \Vert_{t}+\Vert A \Vert_{1-t}\right) \leq 2 \left( \omega^{ }_{t}  (A) 
+ \omega^{ }_{1-t}  (A)  \right).$$

\end{theorem}
\begin{proof}
Let $ \alpha, \beta \in \mathbb{R} $ be such that  $ \alpha^{2}+\beta^{2}=1. $ Then we have

\begin{align*}
&2 \Vert  \alpha \Re(A) + \beta \Im (A)+i(2t-1)(\alpha \Im(A) - \beta \Re (A)) \Vert\\
& = \left\Vert  \alpha  \left( T+(1-2t )A^{*}   \right) + \alpha \left( A^{*}+(2t -1 )A   \right)  -\beta i \left( T+(2 t -1 ) A^{*} \right)  -\beta i \left(  -A^{*}+(2t -1)A  \right) \right\Vert\\
& \leq  \vert \alpha\vert \Vert  A+(1-2t )A^{*}   \Vert +\vert \alpha\vert  \Vert A^{*}+(2t -1 )A   \Vert
+ \vert \beta \vert \Vert A+(2 t -1 ) A^{*}  \Vert + \vert \beta \vert  \Vert  A^{*}+(1 -2t )A \Vert\\
& = \vert \alpha\vert  \Vert A \Vert_{t}+\vert \alpha\vert \Vert A^{*} \Vert_{1-t} + \vert \beta \vert \Vert A \Vert_{1-t}+ \vert \beta \vert  \Vert A^{*} \Vert_{t} \\
&= \left(  \vert \alpha\vert +  \vert \beta \vert  \right) \left( \Vert A \Vert_{t} +  \Vert A \Vert_{1-t} \right)\\
& \leq \sqrt{2} \left( \Vert A \Vert_{t} +  \Vert A \Vert_{1-t} \right)\\
& \leq 2\sqrt{2}\left(   \omega^{ }_{t}  (A) + \omega^{ }_{1-t}  (A) \right).
\end{align*}
This proves the desired inequality.
\end{proof}
\begin{theorem}
Let $A\in\mathcal B(H)$ and $ 0 \leq t \leq 1. $ Then 
$$ \sup_{\theta}\omega^{ }_{t}  (e^{i\theta}A) \leq \Vert A \Vert_{t}  \leq 2 \left(  \omega^{\Re}_{t}  (A) 
+  t \omega(A)\right).$$
In particular, 
$$ \omega^{ }_{t}  (A) \leq \Vert A \Vert_{t}  \leq 2 \left( \omega^{\Re}_{t}  (A) 
+  t \omega(A)\right).$$
\end{theorem}
\begin{proof}
Let $ x \in H $ with   $\Vert x \Vert =1. $ Then\\
\begin{align*}
&\left\vert   \left\langle  \left( \Re_{t}( e^{i\theta} A )+i \Im_{t}( e^{i\theta} A)\right) x, x \right \rangle \right\vert\\
&=\left\vert   \left\langle  \left(  \left( t e^{i\theta} A + (1-t) e^{-i\theta} A^{*}\right)+\left( (1-t) e^{i\theta} A- t e^{-i\theta} A^{*}\right) \right) x, x \right \rangle \right\vert\\
&=\left\vert   \left\langle  \left(  \left( t e^{i \theta} A + (1-t) e^{-i \theta} A^{*}\right)+\left( (1-t) e^{i \theta} A- t e^{-i \theta} A^{*}\right) +t e^{-i \theta}A^{*}-t e^{-i \theta}A^{*} \right) x, x \right \rangle \right\vert\\
&=\left\vert   \left\langle  \left(  \left( t e^{i \theta} A + (1-t) e^{-i \theta} A^{*}\right)+\left( (1-t) e^{i \theta} A+ t e^{-i \theta} A^{*}\right) -2t e^{-i \theta}A^{*} \right) x, x \right \rangle \right\vert\\
&\leq \omega^{\Re}_{t}(A) + \omega^{\Re}_{1-t}  (A) + 2t \omega(A)\\
&=2\omega^{\Re}_{t}(A) +2t \omega(A)\\
\end{align*}
taking the supremum over all $ \theta \in \mathbb{R}, $ we get the desired result.
\end{proof}

Although  $\omega_{(s,t)}$ may appear to be a distinct generalization of  $\omega^{\Re}_t$,  the following proposition shows that the two quantities are equivalent up to a scaling factor.
\begin{proposition}\label{p}
Let $s,t$ be two nonnegative real numbers with $s+t>0$ and $A\in\mathbb{B}(\mathcal H)$. Then
$$\omega_{(s,t)}(A)=(s+t)\omega^{\Re}_{\frac{s}{s+t}}(A).$$
\end{proposition}
\begin{proof}
By a simple calculations
\begin{align*}
\Re_{(s,t)}(A)&=sA+tA^*\\
&=(s+t)(\frac{s}{s+t}A+\frac{t}{s+t}A^*)\\&
=(s+t)\Re_{\frac{s}{s+t}}(A).
\end{align*}
So by definitions of $\omega_{(s,t)}$  and $\omega^{\Re}_t$, the result follows.
\end{proof}

Theorem \ref{4.4} and Theorem \ref{4.5} now follow immediately from Proposition \ref{p}.

Furthermore, we have the following corollary.
\begin{corollary}
Let $s,t$ be two nonnegative real numbers with $s+t>0$ and $A\in\mathbb{B}(\mathcal H)$. Then
\begin{itemize}
\item $\omega_{(s,t)}(A)=\omega_{(t,s)(A)}$.
\item If $A$ is normal, then $(s+t)\omega(A)=\omega_{(s,t)}(A)=(s+t)\|A\|$.
\item $(s+t)\omega(A)\leq \omega_{(s,t)}(A)\leq 2\max\{s,t\}\omega(A)$.
\end{itemize}
\end{corollary}


\begin{thebibliography}{5}
\bibitem{O-F} A. Abu-Omar and F. Kittaneh, \textit{A generalization of the numerical radius }, Linear Algebra Appl., 569 (2019), 323--334.
\bibitem{bot}
T. Bottazzi and C. Conde, \textit{Generalized numerical radius and related inequalities}, Oper. Matrices 15 (2021), no. 4, 1289-1308.
\bibitem{c*}
A. Bourhim and M. Mabrouk, \textit{ $a$-numerical range on C$^*$-algebras}, Positivity 25 (2021), 1489–1510.
\bibitem{sab}
C. Conde, H.R. Moradi and M. Sababheh, \textit{Some weighted numerical radius inequalities}, Indian Journal of Pure and Applied Mathematics, (2022).

 \bibitem{jen} S.S. Dragomir, \textit{Bounds for the normalized Jensen functional }, Bull. Austral. Math. Soc. 3(2006), 471--478.


 \bibitem{he-ha} A. El. Farissi,  \textit{Simple proof and refinement of Hermite-Hadamard inequality}, J. Math. Inequal., 4(3) (2010), 365--369.
 
 

 
 
 
\bibitem{gao}
F. Gao and M.Y. Hou, \textit{A generalization of the weighted algebraic numerical radius on C$^*$-algebras}, Oper. Matrices 19 (2025), no. 1, 31-48.

 %
\bibitem{q-num}
M.I.C. Goncalves and A.R. Sourour, \textit{Isometries of a generalized numerical radius,} Linear Alg. Appl. 429, Issue 7 (2008), 1478-1488.

\bibitem{gust} K. E. Gustafson and D. K. M. Rao,  \textit{Numerical Range}. In: Numerical Range. Universitext. Springer, New York NY. (1997). 


 \bibitem{kit} F. Kittaneh, M. S. Moslehian and T. Yamazaki, {\it Cartesian decomposition and numerical radius inequalities}, Linear Algebra Appl. 471 (2015), 46--53.


\bibitem{mab}
 M. Mabrouk, A. Zamani, \textit{An extension of the a -numerical radius on C$^*$-algebras}, Banach Journal
of Mathematical Analysis 17 (3), 42 (2023).

 %
\bibitem{Sattari} M. Sattari, M. S. Moslehian and T. Yamazaki, \textit{Some generalized numerical radius inequalities for Hilbert space operators}, Linear Algebra Appl, 470 (2014), 1--12.

\bibitem{sheikh}
A. Sheikhhosseini, M. Khosravi and M. Sababheh,\textit{ The weighted numerical radius}, Ann. Funct. Anal. 13, 3 (2022).
 
%
 \bibitem{sheyb} S. Sheybani, M. Sababheh and H. Moradi, {\it Weighted inequalities for the numerical radius}, Vietnam J. Math. 51 (2), (2023), 363-377.


%

%


\bibitem{zamani2}
A.Zamani, \textit{The weighted Hilbert–Schmidt numerical radius}, Linear Alg. Appl. 675 (2023), 225--243.

\bibitem{zamani} A. Zamani, M.S. Moslehian, Q. Xu and C. Fu, \textit{Numerical radius inequalities concerning with algebra norms}, Mediteranean Journal of Mathematics 18 (2019).

 \bibitem{zam_2021} A. Zamani and P. W\'{o}jcik, {\textit {Another generalization of the numerical radius for Hilbert space operators}}, Linear Algebra Appl. 609 (2021), 114-128.
\bibitem{yam} T. Yamazaki, \textit{On upper and lower bounds of the numerical radius and on equality condition}, Studia math. 178, (2007), 83--89.
\end{thebibliography}
\end{document}